\documentclass[12pt]{amsart}
\usepackage{graphicx}
\usepackage[headings]{fullpage}
\usepackage{bm,amssymb,epic,eepic,epsfig,amsbsy,amsmath,amscd}
\numberwithin{equation}{section}
                        \theoremstyle{plain}
\usepackage{mathrsfs}

\newcommand\no[1]{}

\newtheorem{theorem}{Theorem}[section]
\newtheorem{thm}{Theorem}
\newtheorem{lemma}[theorem]{Lemma}

\newtheorem{conjecture}{Conjecture}

\theoremstyle{definition}

\newcommand{\la}{\langle}
\newcommand{\ra}{\rangle}

\def\BC{\mathbb C}

\def\BZ{\mathbb Z}

\def\CA{\mathcal A}

\def\CR{\mathcal R}

\def\CT{\mathcal T}

\def\fp{\mathfrak p}
\def\ft{\mathfrak t}

\def\la{\langle}
\def\ra{\rangle}

\def\ve{\varepsilon}
\def\be { \begin{equation} }
\def\ee { \end{equation} }

\def\BC{\mathbb C}

\def\BZ{\mathbb Z}

\def\CA{\mathcal A}

\def\CR{\mathcal R}

\def\CT{\mathcal T}

\def\fp{\mathfrak p}
\def\ft{\mathfrak t}

\def\la{\langle}
\def\ra{\rangle}

\def\ve{\varepsilon}
\def\be { \begin{equation} }
\def\ee { \end{equation} }

\begin{document}

\title[Strong AJ conjecture for the figure eight knot]
{The strong AJ conjecture for the figure eight knot}

\author{Hoang-An Nguyen}  
\address{Department of Mathematics, The University of Iowa, 
Iowa City, Iowa 52242, USA}
\email{hoang-an-nguyen@uiowa.edu}

\author{Anh T. Tran}
\address{Department of Mathematical Sciences, The University of Texas at Dallas, Richardson, TX 75080, USA}
\email{att140830@utdallas.edu}

\begin{abstract}
Motivated by the theory of quantum A-ideals of Frohman-Gelca-LoFaro, the theory of $q$-holonomicity of quantum invariants of Garoufalidis-Le and the AJ conjecture of Garoufalidis, Sikora formulated the strong AJ conjecture which relates the A-ideal and recurrence ideal of a knot in the 3-sphere. This conjecture has been verified for all torus knots and most of their cables. In this paper, we verify the strong AJ conjecture for the figure eight knot. 
\end{abstract}

\thanks{2000 {\it Mathematics Subject Classification}. Primary 57M27, Secondary 57M25.}

\thanks{{\it Key words and phrases.\/} colored Jones polynomial, A-polynomial, AJ conjecture, figure eight knot.}

\maketitle

%\keywords{colored Jones polynomial, A-polynomial, AJ conjecture, strong AJ conjecture, cable knot, torus knot.}

%\ccode{Mathematics Subject Classification 2010: 57M27, 57N10.}

\section{Introduction}

\subsection{The colored Jones polynomial and recurrence ideal} For a knot $K$ in $S^3$ and a positive integer $n$, let $J_K(n) \in \BZ[t^{\pm 1}]$ denote the $n$-colored Jones polynomial of $K$ with zero-framing. The polynomial $J_K(n)$ is the quantum link invariant, as defined by Reshetikhin and Turaev \cite{RT}, associated to the Lie algebra $sl_2(\BC)$, with the color $n$ standing for the irreducible $sl_2(\BC)$-module $V_{n}$ of dimension $n$. Here we use the functorial normalization, i.e. the one for which $J_{\mathrm{unknot}}(n)=[n] := (t^{2n}- t^{-2n})(t^2 -t^{-2}).$ 
It is known that $J_K(1)=1$ and $J_K(2)$ is the ordinary Jones polynomial \cite{Jones}. The color $n$ can take non-positive integer values by setting $J_K(-n) = - J_K(n)$ and $J_K(0)=0$.

 Consider a discrete function $f: \BZ \to \CR:=\BC[t^{\pm 1}]$ and define the linear operators $L, M$ acting on such functions by $(Lf)(n) := f(n+1)$ and $(Mf )(n) := t^{2n} f(n).$
It is easy to see that $LM = t^2 ML$, and that $L^{\pm 1}, M^{\pm 1}$ generate the quantum torus $\CT$, a non-commutative ring with presentation
$\mathcal T := \mathbb \CR\langle L^{\pm1}, M^{\pm 1} \rangle/ (LM - t^2 ML).$

Let $\mathcal A_K := \{ P \in \mathcal T \mid  P J_K=0\}$ which is a left-ideal of $\mathcal T$, called the {\em recurrence ideal} of $K$. It was proved in \cite{GL} that for every knot $K$, the
recurrence ideal $\mathcal A_K$ is non-zero. An element in $\mathcal A_K$ is called a recurrence relation for the colored Jones polynomial of $K$.

\subsection{The AJ conjecture} The ring $\mathcal T$ is not a principal left-ideal domain, i.e. not every left-ideal of $\mathcal T$ is generated by one element. By adding all  inverses of polynomials in $t,M$ to
$\mathcal T$ one gets a principal left-ideal domain $\tilde\CT$, c.f. \cite{Ga}. The ring $\tilde{\CT}$ can be formally defined as follows. Let
$\CR(M)$ be the fractional field of the polynomial ring $\CR[M]$.
Let $\tilde \CT$ be the set of all Laurent polynomials in the
variable $L$ with coefficients in $\CR(M)$:
$$\tilde
\CT :=\Big\{\sum_{i\in \BZ}a_i(M) L^i \,\, | \quad a_i(M)\in \CR(M),
\,\,\, f_i=0  \quad \text{almost always} \Big\},
$$
and define the product in  $\tilde \CT$ by $a(M) L^{k} \cdot b(M)
L^{l} :=a(M)\, b(t^{2k}M) L^{k+l}.$

The left ideal extension $\tilde \CA_K :=\tilde \CT \CA_K$ of $\CA_K$ in
$\tilde \CT$ is then generated  by a polynomial
$\alpha_K(t,M,L) = \sum_{i=0}^{d} \alpha_{K,i}(t,M) \, L^i$,
where $d$ is assumed to be minimal and all the
coefficients $\alpha_{K,i}(t,M)\in \BZ[t^{\pm1},M]$ are assumed to
be co-prime. That $\alpha_K$ can be chosen to have integer
coefficients follows from the fact that $J_K(n) \in
\BZ[t^{\pm1}]$. The polynomial $\alpha_K$ is defined up to a polynomial in $\mathbb Z[t^{\pm 1},M]$. We call $\alpha_K$ the {\em recurrence polynomial} of $K$.

Let $\ve$ be the map reducing $t=-1.$ Motivated by the theory of quantum A-ideals of Frohman-Gelca-LoFaro \cite{FGL}, Garoufalidis \cite{Ga} formulated the AJ conjecture which relates the A-polynomial and recurrence polynomial of a knot in the 3-pshere.

\begin{conjecture}{\bf (AJ conjecture)} For every knot $K$ in $S^3$, $\ve(\alpha_K)$ is equal to the A-polynomial $A_K(M,L)$, up to multiplication by a polynomial depending on $M$ only.
\end{conjecture}

The A-polynomial of a knot was introduced by Cooper et al. \cite{CCGLS}; it describes the $SL_2(\BC)$-character variety of the knot complement as viewed from the boundary torus. The A-polynomial always contains the factor $L-1$ coming from the abelian component of the character variety, so we can write $A_K = (L-1)A'_K$ where $A'_K \in \BZ[M,L].$

The AJ conjecture has been confirmed for the trefoil knot and figure eight knot \cite{Ga}, all torus knots \cite{Hi, Tr2013}, some classes of two-bridge knots and pretzel knots  \cite{Le, LT, LZ}, the knot $7_4$ \cite{GK}, and certain cable knots \cite{RZ, Ru, Tr2014, Tr2015a, Tr2017}. 

\subsection{Main result} For a finitely generated group $G$, let $\chi(G)$ denote the $SL_2(\BC)$-character variety of $G$, see \cite{CS}. Suppose  $G=\BZ^2$.
Every pair of generators $\mu, \lambda$ defines an isomorphism
between $\chi(G)$ and $(\BC^*)^2/\tau$, where $(\BC^*)^2$ is the
set of non-zero complex pairs $(M,L)$ and $\tau$ is the involution
$\tau(M,L):=(M^{-1},L^{-1})$. For an algebraic set $V$ (over $\BC$), let $\BC[V]$ denote the ring
of regular functions on $V$.  For example, $\BC[(\BC^*)^2/\tau]=
\ft^\sigma$, the $\sigma$-invariant subspace of  $\ft:=\BC[M^{\pm
1},L^{\pm 1}]$, where $\sigma(M^kL^l):= M^{-k}L^{-l}.$

Let $K$ be a knot in $S^3$ and $X:=S^3 \setminus K$ its complement. 
The boundary of
$X$ is a torus whose fundamental group  is free abelian of rank
two. An orientation of $K$ defines a unique pair of an
oriented meridian $\mu$ and an oriented longitude $\lambda$ such that the linking
number between the longitude and the knot is zero. The pair provides
an identification of $\chi(\pi_1(\partial X))$ and $(\BC^*)^2/\tau$
which actually does not depend on the orientation of $K$.

The inclusion $\partial X \hookrightarrow X$ induces an algebra homomorphism
$$\theta: \BC[\chi(\pi_1(\partial X))]  \equiv \ft^\sigma \longrightarrow
\BC[\chi(\pi_1(X))].$$
We call the kernel $\fp$ of $\theta$ the {\em A-ideal} of $K$; it is an ideal of $\ft^\sigma$. The A-ideal was first introduced in \cite{FGL}; it determines the A-polynomial of $K$. In fact $\fp=(A_K \cdot \ft)^{\sigma}$, the $\sigma$-invariant part of the ideal $A_K \cdot \ft \subset \ft$ generated by the A-polynomial $A_K$.

The involution $\sigma$ acts on the quantum torus $\CT$ also by $\sigma(M^kL^l)= M^{-k}L^{-l}$. Let $\CA_K^\sigma$ be the $\sigma$-invariant part of the recurrence ideal $\CA_K$; it is an ideal of $\CT^{\sigma}$. 

Sikora \cite{Si} formulated the following conjecture that relates the A-ideal and recurrence ideal of a knot in the 3-pshere.

\begin{conjecture} {\bf (Strong AJ conjecture)}
Suppose $K$ is a knot in $S^3$. Then $\sqrt{\ve(\CA_K^{\sigma})}=\fp.$
\end{conjecture}

Here $\sqrt{\ve(\CA_K^{\sigma})}$ denotes the radical of the ideal $\ve(\CA_K^{\sigma})$ in the ring $\ft^{\sigma}$. The strong AJ conjecture has been confirmed for the trefoil knot \cite{Si}, all torus knots \cite{Tr2013} and most of their cables \cite{Tr2015b}. In this paper we prove the following.

\begin{thm}
\label{MTheorem}
The strong AJ conjecture holds true for the figure eight knot.
\end{thm}

Let us outline the proof of Theorem \ref{MTheorem}. Suppose the AJ conjecture holds true for a knot $K \subset S^3$ and the A-polynomial $A'_K(M,L)$ does not have any nontrivial factors depending on $M$ only. Then it is known that $\sqrt{\varepsilon(\mathcal{A}^{\sigma}_K)}\subset \mathfrak{p}_K$. For the converse $\sqrt{\varepsilon(\mathcal{A}^{\sigma}_K)}\subset \mathfrak{p}_K$ it suffices to find a symmetric non-homogeneous recurrence polynomial $P \in\mathcal{R}[M^{\pm1},L^{\pm1}]$ such that $\varepsilon(P)=M^k L^l ( A'_K )^{2m}$ for some integers $k,l,m$, see Lemma \ref{Plem} below. The rest of the paper is devoted to constructing such an $P$ for the figure eight knot.  

%%%%%%%%%%%%%%%%%%%%%%%%%%%%%%%%%%%%%%%%%%%%%%%%%%%%%%%%%%%%%%%%%%%%%%%%%%%%%%%%%%%%%%%%%%%%%%%%%%%%%%%%%

\section{Proof of Theorem \ref{MTheorem}} 

For the figure eight knot $E$, it is known that the AJ conjecture holds true \cite{Ga} and the A-polynomial $A'_E = L+L^{-1}-M^{4} - M^{-4}+ M^{2} + M^{-2}+ 2$ does not have any nontrivial nontrivial factors depending on $M$ only. Hence, by \cite[Lemma 3.1]{Tr2015b}, $\sqrt{\varepsilon(\mathcal{A}^{\sigma}_E)}\subset \mathfrak{p}_E$. 

For the converse $\sqrt{\varepsilon(\mathcal{A}^{\sigma}_E)}\subset \mathfrak{p}_E$ we will apply the following.

\begin{lemma}\label{Plem}
Suppose  there exists $P \in\mathcal{R}[M^{\pm1},L^{\pm1}]$ such that $PJ_K\in \mathcal{R}[M^{\pm 1}]$, $\sigma(P)=P$ and $\varepsilon(P)=M^k L^l ( A'_K )^{2m}$ for some integers $k,l,m$. Then $\mathfrak{p}_K\subset \sqrt{\varepsilon(\mathcal{A}^{\sigma}_K)}$.

\begin{proof}
Since $PJ_K \in \mathcal{R}[M^{\pm 1}]$, by \cite[Lemma 3.3]{Tr2015b} there exists $R \in \CR[L^{\pm 1}]$ such that $\sigma(R)=R$, $\varepsilon(R)=(L+L^{-1}-2)^r$ and $RPJ_K=0$ for some integer $r  \ge 1$.

Let $Q=R^{m}P^r \in\mathcal{R}[M^{\pm1},L^{\pm1}]$. Then $QJ_K=0$, $\sigma(Q)=Q$ and $\varepsilon(Q)= \varepsilon(R)^{m} \varepsilon(P)^r = M^{kr}L^{lr-mr} ( (L-1)A'_K )^{2mr}$. Hence, by \cite[Lemma 3.2]{Tr2015b} we obtain $\mathfrak{p}_K\subset \sqrt{\epsilon(\mathcal{A}^{\sigma}_K)}$.
\end{proof}
\end{lemma}

By Lemma \ref{Plem} we will  find $P \in\mathcal{R}[M^{\pm1},L^{\pm1}]$ such that $PJ_E\in \mathcal{R}[M^{\pm 1}]$, $\sigma(P)=P$ and $\varepsilon(P)= ( A'_E )^{2}$. Note that 
$$(A'_E)^2=\tilde a_{2}(M)L^2 + \tilde a_{-2}(M)L^{-2} + \tilde a_1(M)L + \tilde a_{-1}(M)L^{-1} + \tilde a_0(M),$$
where
\begin{eqnarray*}
\tilde a_2(M) &=& \tilde a_{-2}(M)  = 1,\\
\tilde a_1(M) &=& \tilde a_{-1}(M) = -2 M^{4} - 2 M^{-4} + 2 M^{2} + 2 M^{-2} + 4,\\
\tilde a_0(M) &=&M^{8} + M^{-8} - 2 M^{6} - 2 M^{-6} - 3 M^{4} - 3 M^{-4} + 2 M^{2} + 2 M^{-2} + 10.
\end{eqnarray*}

To find $P$, we  start with a non-homogeneous recurrence polynomial for $J_E$. Let 
$$
\alpha'_E(t,M,L)=a_1(t,M)L+a_{-1}(t,M)L^{-1}+a_0(t,M)
$$
where
\begin{eqnarray*}
a_1(t,M)&=& t^{-2} M^2 - t^2 M^{-2},\\
a_{-1}(t,M)&=& t^2 M^2 - t^{-2} M^{-2},\\
a_0(t,M)&=& (M^2 - M^{-2}) (-M^4 - M^{-4} + M^2 + M^{-2} + t^4 + t^{-4}).
\end{eqnarray*}
By \cite[Proposition 4.4]{CM}  we have $\alpha'_E J_E\in \CR[M^{\pm 1}]$.

Observe that if $B(t, M, L)$ is a rational function such that $B\alpha'_E$ is a Laurent polynomial in $t, M, L$ then $B\alpha'_E$ is another non-homogeneous recurrence polynomial for $J_E$. Namely $B \alpha'_E J_E\in \CR[M^{\pm 1}]$. This suggests that we may consider $P \in\mathcal{R}[M^{\pm1},L^{\pm1}]$ of the form $P=B \alpha'_E$ where 
$$
B(t,M,L) = b_1(t,M)L+b_{-1}(t,M)L^{-1}+b_0(t,M)
$$
and $b_i$'s are \textit{rational functions} in $t$ and $M$. Explicitly we have 
$$
P(t, M, L) = p_2(t,M)L^2+p_{-2}(t,M)L^{-2}+p_1(t,M)L+p_{-1}(t,M)L^{-1}+p_{0}(t,M)
$$
where
\begin{eqnarray*}
p_2(t,M) &=& b_1(t, M) a_1(t, t^2 M), \\
p_{-2}(t,M) &=& b_{-1}(t, M) a_{-1}(t, t^{-2} M), \\
p_{1}(t,M) &=& b_0(t, M) a_1(t, M) + b_1(t, M) a_0(t, t^2 M), \\
p_{-1}(t,M) &=& b_0(t, M) a_{-1}(t, M) + b_{-1}(t, M) a_0(t, t^{-2} M), \\
p_0(t,M) &=& b_0(t, M) a_0(t, M) + b_{-1}(t, M) a_1(t, t^{-2} M) + b_1(t, M) a_{-1}(t, t^2 M).
\end{eqnarray*}

Since $P J_K = B\alpha'_E J_K \in \mathcal{R}[M^{\pm 1}]$,  the remaining conditions for $P$  are $\sigma(P)=P$ and $\varepsilon(P)=( A'_E )^{2}$. These conditions are equivalent to the followings
\begin{eqnarray}
\label{eq: e2}
p_2(t,M)&=&p_{-2}(t,M^{-1}),\\
\label{eq: e1}
p_1(t,M)&=&p_{-1}(t,M^{-1}),\\
\label{eq: e0}
p_0(t,M)&=&p_{0}(t,M^{-1}),
\end{eqnarray}
and
\begin{eqnarray}
p_2(1,M)&=& \tilde{a}_{2}(M),  \label{eq: 2}\\
p_{-2}(1,M)&=&\tilde{a}_{-2}(M),  \label{eq: -2} \\
p_1(1,M)&=& \tilde{a}_{1}(M), \label{eq: 1}\\
p_{-1}(1,M)&=& \tilde{a}_{-1}(M) \label{eq: -1},\\
p_0(1,M)&=& \tilde{a}_{0}(M). \label{eq}
\end{eqnarray}

We will call the first set of conditions the \textit{symmetric conditions} and the second set the \textit{initial conditions}. We also note that $p_i$'s must be polynomials in $t, M, L$, and we call these conditions the \textit{polynomial conditions}.  We now give a detailed method for solving the system of equations \eqref{eq: e2}--\eqref{eq}.

\subsection{Step 1: Applying symmetric conditions}
We first choose the leading coefficients $p_2$ and $p_{-2}$ to be the same as those of the A-polynomial squared. Namely we let
$$p_2(t,M)= \tilde{a}_2(M)=1 \hspace{4mm}\text{and}\hspace{4mm} p_{-2}(t,M)=\tilde{a}_{-2}(M)=1.$$
We see that this takes care of equation \eqref{eq: e2}. Then
$$
b_1(t,M) = \frac{1}{a_1(t,t^2 M)}  \hspace{4mm}\text{and}\hspace{4mm}
b_{-1}(t,M) = \frac{1}{a_{-1}(t,t^{-2} M)}. 
$$
This implies that 
\begin{eqnarray}
p_1(t,M) &=& a_1(t, M) b_0(t, M) + \frac{a_0(t, t^2 M)}{a_1(t, t^2 M)},  \label{eq:q1}\\
p_{-1}(t,M) &=&  a_{-1}(t, M) b_0(t, M) +  \frac{a_0(t, t^{-2} M) }{a_{-1}(t,t^{-2} M)}, \\
p_0(t,M) &=& a_0(t, M) b_0(t, M)  + \frac{a_1(t, t^{-2} M) }{a_{-1}(t,t^{-2} M)} +  \frac{a_{-1}(t, t^2 M)}{a_1(t,t^2 M)}. \label{eq:q0}
\end{eqnarray}

Since $a_1(t,M) = - a_{-1}(t,M^{-1})$ and $a_0(t,M) = - a_0(t,M^{-1})$, we see that equations $p_1(t,M)=p_{-1}(t,M^{-1})$ and $p_0(t,M)=p_0(t,M^{-1})$ are equivalent to a single equation
\be\label{eq: b0}
b_0(t,M^{-1})=-b_0(t,M).
\ee
Hence conditions \eqref{eq: e1} and \eqref{eq: e0} become equation \eqref{eq: b0}.

\subsection{Step 2: Applying polynomial conditions}

We now use the fact that $p_i$'s must be Laurent polynomials. This imposes additional conditions on $b_0$, but it is hard to work with rational equations in terms of $b_0$. So  we will work with equations in terms of $p_i$'s directly. To do this, we first express $b_0$ in terms of $p_i$'s.

By solving for $b_0$ from equation \eqref{eq:q1} we have
$$b_0(t,M)=-\frac{a_0(t,t^2 M)-a_1(t,t^2 M) p_1(t,M)}{a_1(t,M)a_1(t,t^2 M)}.$$

Substituting this into equation \eqref{eq:q0}, we obtain the following Diophantine equation in a polynomial ring
\begin{equation}
\label{eq:Dio1}
A_1(t,M)p_0(t,M)+B_1(t,M)p_1(t,M)=C_1(t,M)    
\end{equation}
where 
\begin{eqnarray*}
A_1(t,M)&=&-a_1(t, M) a_1(t,  t^2 M) a_{-1}(t, t^{-2}M),\\
B_1(t,M)&=&a_0(t, M) a_1(t,  t^2 M) a_{-1}(t,  t^{-2} M),\\
C_1(t,M)&=&-a_1(t, M) a_1(t, t^{-2} M) a_1(t,  t^2 M) +a_0(t, M) a_0(t,  t^2 M) a_{-1}(t, t ^{-2} M) \\
&&- \, a_1(t, M) a_{-1}(t, t^{-2} M) a_{-1}(t,  t^2 M).
\end{eqnarray*}
Note that in arriving at this equation, we did not use  $b_0(t,M^{-1})=-b_0(t,M)$. This means that we will have to implement this condition later.

\subsection{Step 3: Checking if  equation \eqref{eq:Dio1} has a solution}

We see that  equation \eqref{eq:Dio1}
is satisfied if and only if $C_1(t,M)$ is in the ideal  generated by $A_1(t,M)$ and $B_1(t,M)$. This type of problem is known as the ideal membership problem. One way to solve the ideal membership problem is to use Groebner bases.

For calculation purposes we raise the powers of $t$ and $M$. We do this by multiplying by factors of $t$ and $M$. Let
\begin{eqnarray*}
A_2(t,M)&=& t^8M^{10} A_1(t,M),\\
B_2(t,M)&=& t^8 M^{10} B_1(t,M),\\
C_2(t,M)&=& t^8M^{10} C_1(t,M).
\end{eqnarray*}
By a direct calculation we can find a Groebner basis of the ideal generated by $A_2(t,M)$ and $B_2(t,M)$. Indeed we use the following command in Mathematica
$$\{\mathrm{gb}, \mathrm{mat}\} = 
 \mathrm{GroebnerBasis`BasisAndConversionMatrix}[{A_2(t,M), B_2(t,M)}, {t, M}, \{\}]$$
where $\mathrm{gb}$ is a Groebner basis obtained from $A_2,B_2$, and $\mathrm{mat}$ is the matrix that allows us to write the Groebner basis in terms of $A_2,B_2$. Namely 
$$\mathrm{mat}.\{A_2,B_2\}-\mathrm{gb}=0.$$

Explicitly, we have the following Groebner basis $\mathrm{gb}=\{g_1,g_2,g_3,g_4\}$ for $\la A_2, B_2 \ra$
\begin{eqnarray*}
g_1&=&-M^{14} (-1 + M^8) (t^4 + M^8 t^4 - M^4 (1 + t^8)), \\
g_2&=&-M^{10} (-1 + M^8) t^2 (t^4 + M^8 t^4 - M^4 (1 + t^8)), \\
g_3&=&(-1 + M^4)  (M - t) (M + t) (-1 + M t) (1 + M t) (M^2 + t^2) (1 + M^2 t^2) \\
   &&\times (M^4 + M^6 + M^8 + 2 M^{10} + 3 M^{12} + M^{14} + 2 M^{16} - t^4), \\
g_4&=&t^2 (-M + t) (M + t) (-1 + M t) (1 + M t) (M^2 + t^2) (1 + M^2 t^2)\\
   &&\times(-M^4 - M^6 - M^8 + M^{14} + M^{16} + t^4).
\end{eqnarray*}
For the matrix $\mathrm{mat} = \{m_{ij}\}$ we have 
\begin{eqnarray*}
m_{11}&=& (-1 + M^4)  t^2 \times(1 - 2 M^2 + M^4 - 2 M^6 + 2 M^8 - 
   M^{10} \\
   &&\qquad \qquad \qquad \qquad \qquad \quad \,+ \, M^{12} - M^{14} - M^4 t^4 + M^6 t^4 + M^{10} t^4),\\
m_{12}&=& -M^4 (M^6 + M^{10} + t^4 - M^2 t^4 - M^6 t^4),\\
m_{21}&=& -(-1 + M^4)  (1 - M^2 - M^6 - M^8 - M^6 t^4 - 
   M^{10} t^4 + M^{12} t^4 - M^8 t^8),\\
m_{22}&=& M^4 t^2 (-1 + M^8 - M^4 t^4),\\
m_{31}&=& (-1 + M^4)  t^2 \times(-4 + M^2 - 2 M^4 + 3 M^6 - 2 M^8 + 
   M^{10} \\
   && \qquad \qquad \qquad \qquad \qquad \quad \, \, \,- \, 2 M^{12} + 3 M^4 t^4 + M^6 t^4 + 2 M^8 t^4), \\
m_{32}&=& 1 + M^2 + M^4 + 2 M^6 + 3 M^8 + M^{10} + 2 M^{12} - 3 M^4 t^4 - M^6 t^4 - 
 2 M^8 t^4, \\
m_{41}&=& -2 - M^2 - M^4 + 2 M^8 + M^{10} - M^2 t^4 - M^4 t^4 + M^8 t^4 + 
 2 M^{10} t^4 \\ 
 && - \, M^{14} t^4 - M^4 t^8 - M^6 t^8 + M^8 t^8 + M^{10} t^8,\\
m_{42}&=& -(1 + M^2) t^2 (1 - M^8 + M^4 t^4).
\end{eqnarray*}

Now to check if $C_2$ is in the deal $\la A_2,B_2 \ra$  we use the command
$$\{q,r\}=\mathrm{PolynomialReduce}[C_2, \mathrm{gb},{t,M}].$$
One would see that $r=0$, which means that
$$C_2=q_1 g_1 + q_2 g_2 + q_3 g_3 + q_4 g_4$$
and hence $C_2$ is in the ideal $\la A_2,B_2 \ra$. The explicit values of $q_i$'s are
\begin{eqnarray*}
q_1 &=& M^{-4} t^{-10}(3 + M^2 + 2 M^4 + t^4 - M^2 t^4 - M^4 t^4 - 2 M^6 t^4 - 4 t^8 - 
  6 M^2 t^8 \\
  &&- 7 M^4 t^8 - 7 M^6 t^8 - 4 M^8 t^8 - 4 M^{10} t^8 - 
  5 t^{12} - 7 M^2 t^{12} - 6 M^4 t^{12}- 7 M^6 t^{12} \\
  &&- 2 M^8 t^{12} - 4 t^{16} - 6 M^2 t^{16} - 3 M^4 t^{16} - 5 M^6 t^{16} - M^8 t^{16} - 
  2 M^{10} t^{16} - t^{20} \\
  &&- 2 M^2 t^{20} - 2 M^4 t^{20} - 3 M^6 t^{20} - 
  2 M^8 t^{20} + M^4 t^{24} + M^6 t^{24} + M^8 t^{24} + 2 M^{10} t^{24}),\\
q_2 &=& -M^{-4} t^{-8}(-1 - M^2 + 4 t^4 + 8 M^2 t^4 + 4 t^8 + 6 M^2 t^8 + 4 t^{12} + 
 7 M^2 t^{12} - t^{20}),\\
q_3 &=& M^{-4} t^{-10}(-1 + M^2 + M^6 + M^2 t^4 - M^8 t^4 - 2 t^8 - M^4 t^8 - M^6 t^8 - 
  2 M^8 t^8 \\
  &&- M^{10} t^8- 2 M^{12} t^8 - 2 t^{12} - M^4 t^{12} - M^6 t^{12} - 
  3 M^8 t^{12} - M^{10} t^{12} - 2 t^{16} \\
  && - 2 M^4 t^{16}- 2 M^8 t^{16} - 
  M^{12} t^{16} - M^8 t^{20} - M^{10} t^{20} + M^{12} t^{24}),\\
q_4 &=& -2 M^{-4} t^{-4}(1 - t + t^2) (1 + t + t^2) (1 - t^2 + t^4).
\end{eqnarray*}
 
\subsection{Step 4: Finding solutions of equation \eqref{eq:Dio1} }

We see that  equation \eqref{eq:Dio1} 
is reminiscent of a Diophantine equation.
Recall for $a,b,c,x,y\in \BZ$, the equation
$$ax+by=c$$
has the following solutions
$$
x = x_0 + \frac{b}{\gcd(a,b)}k, \qquad
y=y_0 - \frac{a}{\gcd(a,b)}k,
$$
where $(x_0,y_0)$ is a particular solution of the equation, $k\in \BZ$, and $\gcd(a,b)$ is the greatest common divisor of $a,b$.

We will apply the same idea to our problem, except that our ring is a polynomial ring. Given a particular solution $(\tilde p_0, \tilde p_1)$ of equation \eqref{eq:Dio1}, a family of solutions is given by
\begin{eqnarray*}
p_0(t,M)&=&\tilde p_0(t, M) - \frac{f(t, M) B_2(t, M)}{\gcd(A_2,B_2)},\\
p_1(t,M)&=&\tilde p_1(t, M) + \frac{f(t, M) A_2(t, M)}{\gcd(A_2,B_2)},
\end{eqnarray*}
where 
$\gcd(A_2,B_2)$ is the greatest common divisor between $A_2,B_2$, and $f$ is an arbitrary polynomial. 
To find a particular solution $(\tilde p_0, \tilde p_1)$, we express $C_2$ in terms of the Groebner basis $\mathrm{gb}$ and then we express the Groebner basis in terms of $A_2,B_2$ using the matrix $\mathrm{mat} = \{m_{ij}\}$. Indeed, recall that $C_2=q_1 g_1 + q_2 g_2 + q_3 g_3 + q_4 g_4$. Moreover
 \begin{eqnarray*}
g_1 &=& m_{11} A_2 + m_{12} B_2,\\
g_2 &=& m_{21} A_2 + m_{22} B_2,\\
g_3 &=& m_{31} A_2 + m_{32} B_2,\\
g_4 &=& m_{41} A_2 + m_{42} B_2.
\end{eqnarray*}
Hence $C_2 = A_2 \tilde p_0+B_2 \tilde p_1$
where
\begin{eqnarray*}
\tilde p_0 &=& q_1 m_{11}+q_2 m_{21} + q_3 m_{31} + q_4 m_{41},\\
\tilde p_1 &=& q_1 m_{12}+q_2 m_{22} + q_3 m_{32} + q_4 m_{42}.
\end{eqnarray*}
Explicitly, we have
\begin{eqnarray*}
\tilde p_0 &=&\frac{1}{M^4 t^{10}} (-M^4 t^2 - M^6 t^2 - M^8 t^2 + 2 M^{12} t^2 + M^{14} t^2 + t^6
   - M^2 t^6 \\
   && + \, 14 M^4 t^6 + 5 M^6 t^6 - 2 M^8 t^6 + 3 M^{10} t^6 - 
     10 M^{12} t^6 - 6 M^{14} t^6 - M^{16} t^6 \\
     && - \, M^{18} t^6 - 12 t^{10} + 
     6 M^2 t^{10} + 18 M^4 t^{10} + 17 M^6 t^{10} + 6 M^8 t^{10}- 12 M^{10} t^{10} \\
     && - \, 14 M^{12} t^{10} - 19 M^{14} t^{10} + 4 M^{16} t^{10} + 8 M^{18} t^{10} - 13 t^{14}+ 6 M^2 t^{14} \\
     && + \,  35 M^4 t^{14} + 23 M^6 t^{14} + 
     2 M^8 t^{14} - 10 M^{10} t^{14} - 26 M^{12} t^{14} -  25 M^{14} t^{14} \\
     && + \,  4 M^{16} t^{14} + 6 M^{18} t^{14} - 12 t^{18} + 6 M^2 t^{18} + 20 M^4 t^{18} + 17 M^6 t^{18} +  3 M^8 t^{18} \\
     && - \,  12 M^{10} t^{18} - 15 M^{12} t^{18} - 
     18 M^{14} t^{18} + 4 M^{16} t^{18} + 7 M^{18} t^{18} - M^2 t^{22} \\
     &&  + \,  13 M^4 t^{22} + 7 M^6 t^{22} - 3 M^8 t^{22} - 10 M^{12} t^{22} - 6 M^{14} t^{22} + 2 M^4 t^{26} -  M^6 t^{26} \\
     && - \, 2 M^8 t^{26} - M^{10} t^{26} + 
     2 M^{12} t^{26} + 2 M^{14} t^{26} - 2 M^{16} t^{26} - 2 M^8 t^{30}+ 2 M^{12} t^{30}), \\
\tilde p_1&=&\frac{1}{M^4 t^{10}}
   (-1 + M^2 t^4 + M^8 t^4 + M^{10} t^4 + 2 M^{12} t^4 + M^{14} t^4 + 
     M^6 t^8\\
     && - \, 14 M^8 t^8 - 7 M^{10} t^8 - 12 M^{12} t^8 - 8 M^{14} t^8 + 
     13 M^4 t^{12} + 6 M^6 t^{12} - M^8 t^{12} \\
     && + \, 2 M^{10} t^{12} - 10 M^{12} t^{12} - 6 M^{14} t^{12} + 13 M^4 t^{16} + 7 M^6 t^{16} - 2 M^8 t^{16} \\
     &&- \, 11 M^{12} t^{16} - 7 M^{14} t^{16} + 12 M^4 t^{20} + 
     6 M^6 t^{20} + 10 M^8 t^{20} + 6 M^{10} t^{20} \\
     && + \, M^6 t^{24} + 2 M^{12} t^{24} - 2 M^8 t^{28}).
\end{eqnarray*}

\subsection{Step 5: Imposing condition \eqref{eq: b0} on solutions}

We now implement the condition $b_0(t,M^{-1})=-b_0(t,M)$. Since $$b_0(t,M)=-\frac{a_0(t,t^2 M)-a_1(t,t^2 M) p_1(t,M)}{a_1(t,M)a_1(t,t^2 M)},$$ we see that $b_0(t,M^{-1})=-b_0(t,M)$ is equivalent to
\begin{eqnarray*}
&& (M^8 t^{14} - M^4 t^{10}) p_1(t, M) - (M^8 t^{10} - M^4 t^{14}) p_1(t, M^{-1}) \\
&=&
  1 + M^{12} - M^2 t^4 - M^4 t^4 - M^8 t^4 - M^{10} t^4 - M^4 t^8 - 
   M^8 t^8 + M^4 t^{16} \\
   && + \,  M^8 t^{16} + M^2 t^{20} + M^4 t^{20} + M^8 t^{20} + 
   M^{10} t^{20} - t^{24} - M^{12} t^{24}.
\end{eqnarray*}

By replacing $p_1(t,M)=\tilde p_1(t,M)+\frac{f(t,M)A_2(t,M)}{\gcd(A_2, B_2)}$, the above equation is equivalent to
\begin{eqnarray}\label{eq:f}
&&
M^{14} t^8 f(t, M) + M^2 t^8 f(t, M^{-1}) \\
&=& \nonumber
-1 - 2 M^2 - M^4 - M^6 - M^{10} - M^{12} - 2 M^{14} - M^{16}  \nonumber\\  
    && + \, 8 t^4 + 12 M^2 t^4 + 6 M^4 t^4 + 12 M^6 t^4
    - 2 M^8 t^4 + 12 M^{10} t^4 + 6 M^{12} t^4 + 12 M^{14} t^4  \nonumber\\
    && + \,  8 M^{16} t^4 + 6 t^8 + 10 M^2 t^8 + 6 M^4 t^8 + 13 M^6 t^8 + 13 M^{10} t^8 + 6 M^{12} t^8 + 
    10 M^{14} t^8  \nonumber\\
    && + \,  6 M^{16} t^8 + 7 t^{12} + 11 M^2 t^{12} + 6 M^4 t^{12} + 
    12 M^6 t^{12} + 12 M^{10} t^{12} + 6 M^{12} t^{12} \nonumber\\
    && + \, 11 M^{14} t^{12} + 7 M^{16} t^{12} + M^4 t^{16} + M^{12} t^{16} - 2 M^2 t^{20} - 2 M^{14} t^{20}. \nonumber
\end{eqnarray}

\subsection{Step 6: Applying initial  conditions} Now we take into account the initial conditions \eqref{eq: 2}--\eqref{eq}. Firstly, we note that equations \eqref{eq: 2} and \eqref{eq: -2} hold true since we chose $p_2(t,M)=  p_{-2}(t,M)=1.$ Secondly, by taking $t=1$ in $p_1(t,M) = p_{-1}(t,M^{-1})$ we have $p_1(1,M) = p_{-1}(1,M^{-1})$. This implies that $p_1(1,M) = \tilde{a}_{1}(M)$ if and only if $p_{-1}(1,M)= \tilde{a}_{-1}(M)$. Meaning equations  \eqref{eq: 1} and \eqref{eq: -1} are equivalent. 

Thirdly,  by taking $t=1$ in equation \eqref{eq:Dio1} we have 
$$A_1(1,M)p_0(1,M)+B_1(1,M)p_1(1,M)=C_1(1,M).$$
By a direct calculation we can check that  $A_1(1,M)\tilde{a}_{0}(M)+B_1(1,M) \tilde{a}_{1}(M)=C_1(1,M)$. This implies that $p_1(1,M) = \tilde{a}_{1}(M)$ if and only if
$p_0(1,M)= \tilde{a}_{0}(M)$. Meaning equations  \eqref{eq: 1} and \eqref{eq} are equivalent. Hence we have shown that the initial conditions \eqref{eq: 2}--\eqref{eq} reduce to a single equation $p_1(1,M) = \tilde{a}_{1}(M)$. 

Since $$p_1(1,M)=\tilde p_1(1,M)+\frac{f(1,M)A_2(1,M)}{\gcd(A_2(1,M), B_2(1,M))},$$ by a direct calculation we see that the condition $p_1(1,M) = \tilde a_1(M)$ is equivalent to $f(1, M) = M^{-8} - M^{-6} + 35 M^{-4} + 18 M^{-2} + 29 + 20 M^{2}$. This suggests that we may consider the polynomial $f(t,M)$ in equation \eqref{eq:f} of the following form
$$f(t, M) = M^{-8} - M^{-6} + 35 M^{-4} + 18 M^{-2} + 29 + 20 M^{2} + (t^2 - 1) h(t, M)$$
for some  unknown function $h(t,M)$. Now it remains to find $h(t,M)$ explicitly.

With the change of variable $h(t, M) = t^8 M^{-6} k(t, M)$, equation \eqref{eq:f} becomes  
\begin{eqnarray*}
&& k(t, M) + k(t, M^{-1}) \\
&=&  2 t^4 + 2 t^6 + (1 + t^2 - 11 t^4 - 11 t^6 + 12 t^8 + 12 t^{10}) (M^2 +M^{-2}) \\ 
 && + \, (1 + t^2 - 5 t^4 - 5 t^6 + 7 t^8 + 7 t^{10} + t^{12} + t^{14}) (M^4 + M^{-4}) 
 \\
 && + \, (2 + 2 t^2 - 10 t^4 - 10 t^6 + 9 t^8 + 9 t^{10} - 2 t^{12} - 2 t^{14} - 
  2 t^{16} - 2 t^{18}) (M^6 + M^{-6}) \\
  && + \,  (1 + t^2 - 7 t^4 - 7 t^6 + 7 t^8 + 7 t^{10}) (M^8 + M^{-8}).
\end{eqnarray*}
It is easy to see that a solution of this equation is
\begin{eqnarray*}
k(t, M) 
&=&   t^4 +  t^6 + (1 + t^2 - 11 t^4 - 11 t^6 + 12 t^8 + 12 t^{10}) M^2  \\ 
 && + \, (1 + t^2 - 5 t^4 - 5 t^6 + 7 t^8 + 7 t^{10} + t^{12} + t^{14}) M^4 
 \\
 && + \, (2 + 2 t^2 - 10 t^4 - 10 t^6 + 9 t^8 + 9 t^{10} - 2 t^{12} - 2 t^{14} - 
  2 t^{16} - 2 t^{18}) M^6 \\
  && + \,  (1 + t^2 - 7 t^4 - 7 t^6 + 7 t^8 + 7 t^{10}) M^8.
\end{eqnarray*}

With this we have an explicit solution $f(t,M)$ and hence an explicit solution $b_0$. More precisely, we have
$$
b_0(t, M) = -\frac{(-1 + M) (1 + M) (1 + M^2) t^4 (-M^2 - M^4 - M^6 + 
     t^4 + M^8 t^4 - M^4 t^8)}{
  M^2 (M - t) (M + t) (-1 + M t) (1 + M t) (M^2 + t^2) (1 + M^2 t^2)}.
$$

\subsection{Step 7: Explicit solution} The polynomial we are constructing is
$$P(t, M, L)=L^2+L^{-2}+p_1(t,M)L+p_{-1}(t,M)L^{-1}+p_{0}(t,M)$$
where
\begin{eqnarray*}
p_1(t,M) &=& a_1(t, M) b_0(t, M) + \frac{a_0(t, t^2 M)}{a_1(t, t^2 M)},  \\
p_{-1}(t,M) &=&  a_{-1}(t, M) b_0(t, M) +  \frac{a_0(t, t^{-2} M) }{a_{-1}(t,t^{-2} M)}, \\
p_0(t,M) &=& a_0(t, M) b_0(t, M)  + \frac{a_1(t, t^{-2} M) }{a_{-1}(t,t^{-2} M)} +  \frac{a_{-1}(t, t^2 M)}{a_1(t,t^2 M)}.
\end{eqnarray*}
Having now an explicit expression for $b_0$, we obtain 
\begin{eqnarray*}
p_1(t,M)&=&\frac{-1 - t^{16} + M^4 t^8 (1 + t^4)^2 + M^2 (t^4 + t^{12}) + 
 M^6 (t^{8} + t^{16}) - M^8 (t^{12} + t^{20})}{M^{4}t^{10}}, \\
 p_{-1}(t,M)&=&\frac{M^4 t^8 (1 + t^4)^2 - t^{12} (1 + t^8) + M^6 (t^4 + t^{12}) - M^8 (1 + t^{16}) + M^2 (t^8 + t^{16})}{M^4t^{10}}, \\
p_0(t,M)&=&\frac{1}{M^8t^4}\big(t^8 + M^{16} t^8 - M^4 t^8 (2 + t^4) - M^{12} t^8 (2 + t^4) - 
 M^2 (t^4 + t^8)\\
 && + \, M^6 (t^4 + t^8) + M^{10} (t^4 + t^8) - M^{14} (t^4 + t^8) + 
 2 M^8 (1 + t^4 + 2 t^8 + t^{12})\big).
\end{eqnarray*}

Then the polynomial $P \in\mathcal{R}[M^{\pm1},L^{\pm1}]$ constructed above satisfies $PJ_E\in \mathcal{R}[M^{\pm 1}]$, $\sigma(P)=P$ and $\varepsilon(P)= ( A'_E )^{2}$. We will also verify $PJ_E\in \mathcal{R}[M^{\pm 1}]$  directly in the next subsection. Hence, by Lemma \ref{Plem} we have $\sqrt{\varepsilon(\mathcal{A}^{\sigma}_E)}\subset \mathfrak{p}_E$. This completes the proof of the strong AJ conjecture for the figure eight knot. 

\subsection{Double-checking the recurrence relation}
 Following \cite{CM}, for $Q\in \BC[t,M,x]$ we define $\left\langle  Q\right\rangle: \BZ \rightarrow \mathcal{R}=\BC[t^{\pm 1}]$ by
$$\left\langle Q \right\rangle(n):= [n]\sum^{n-1}_{k=0}P(t,t^{2n},t^{4k})\prod^{k}_{l=1}(t^{4n}+t^{-4n}-t^{4l}-t^{-4l})$$
for $n\geq 1$, and $\left\langle Q \right\rangle(-n)=-\left\langle Q \right\rangle(n)$ for $n\leq 0$.
It is known that the colored Jones polynomial of the figure eight knot is $J_E = \la 1 \ra$. 

For $Q, R\in \BC[t,M,x]$, we write $Q\equiv R$ if $Q-R\in \BC[t,M]$. Then we have the following.
\begin{lemma}\label{ALemma}
\begin{eqnarray*}
L\left\langle 1\right\rangle &\equiv& \left( t^{-2}M^{-4}-t^{-2}M^{-2}-t^2 \right)\left\langle 1\right\rangle 
+ \left( t^6M^2-t^2M^{-2} \right)\left\langle x\right\rangle,\\
L^2\left\langle 1\right\rangle &\equiv&
\big\{ 
t^{-12}M^{-8} 
- \left(t^{-12} + t^{-8}\right)M^{-6} 
- t^{12}M^4 - (t^{-4}+1)M^{-4}\\
&&  \qquad \qquad \,\,\, + \, \left(t^{-4} + 1\right)M^{-2} +t^8 + t^4 +1
\big\} \left\langle 1\right\rangle \\
&&  + \, \big \{
t^{16}M^6 - t^{-8}M^{-6} - t^{12}M^{4} + t^{-4}M^{-4} 
- \left(t^{12}+t^8+t^4 \right)M^2 \\
&&  \qquad \qquad \,\, + \,\left(t^4 + t^{-4} + 1\right)M^{-2}
\big \}\left\langle x\right\rangle, \\
L^{-1}\left\langle 1\right\rangle &\equiv&
\left( t^{-2}M^4 - t^{-2}M^2 -t^{2}\right)\left\langle 1\right\rangle
+\left(t^6M^{-2} - t^2 M^2 \right)\left\langle x\right\rangle, \\
L^{-2}\left\langle 1\right\rangle &\equiv&
\big\{
t^{-12}M^8 - \left(t^{-12} + t^{-8} \right)M^6 - t^{12}M^{-4} - \left(t^{-4} + 1\right)M^4\\
&& \qquad \qquad  + \,(t^{-4}+1)M^2 + t^{8} + t^4 +1
\big\}\left\langle 1\right\rangle\\
&& + \,  \big\{
t^{16}M^{-6} - t^{-8}M^6 - t^{12}M^{-4} + t^{-4}M^{4} - \left(t^{12}+t^8+t^4\right)M^{-2}\\
&& \qquad \qquad \,\,\,\,\, + \,\left(t^4+t^{-4}+1\right)M^{2}
\big\}\left\langle x\right\rangle.
\end{eqnarray*}
\end{lemma}
\begin{proof}
From \cite[Proposition 4.5]{CM} we have 
\begin{eqnarray*}
L\left\langle 1\right\rangle &\equiv& \left( t^{-2}M^{-4}-t^{-2}M^{-2}-t^2 \right)\left\langle 1\right\rangle 
+ \left( t^6M^2-t^2M^{-2} \right)\left\langle x\right\rangle, \\
L\la x \ra &\equiv& (t^{-2}M^{-2}-t^2M^2)\la 1 \ra +(t^6M^4-t^2M^2-t^2)\la x \ra.
\end{eqnarray*}
By a direct calculation, using the fact that $L^2\la 1 \ra=L(L\la 1 \ra)$ and the above formulas, we obtain the formula for $L^2\la 1 \ra$. The proofs for $L^{-1}\left\langle 1\right\rangle$ and $L^{-2}\left\langle 1\right\rangle$ are similar. 
\end{proof}

Using Lemma \ref{ALemma} and a direct calculation we see that
$$
(L^2 + L^{-2} + p_1(t,M)L + p_{-1}(t,M)L^{-1} +p_0(t,M)) \la 1 \ra \equiv 0.
$$
This means that $PJ_E\in \mathcal{R}[M^{\pm 1}]$. 

\section*{Acknowledgements} 
The second author has been supported by a grant from the Simons Foundation (\#354595).

\end{document}